\documentclass[12pt,reqno]{amsart}
\usepackage{amsthm,amsmath,amssymb,amscd}
\usepackage{stmaryrd}

\input xy
\xyoption{all}
\CompileMatrices

\usepackage{anysize} \marginsize{1.3in}{1.3in}{1in}{1in}

{
   \newtheorem{theorem}{Theorem}[section]
   \newtheorem{prop}[theorem]{Proposition}
   \newtheorem{lemma}[theorem]{Lemma}

   \newtheorem{corollary}[theorem]{Corollary}

}
{\theoremstyle{definition}
   
   \newtheorem{example}[theorem]{Example}
   \newtheorem{definition}[theorem]{Definition}
   \newtheorem{remark}[theorem]{Remark}

}

\newcommand{\N}{{\mathbb{N}}}

\newcommand{\Z}{{\mathbb{Z}}}

\newcommand{\A}{{\mathbb{A}}}

\renewcommand{\O}{\mathcal{O}}

\newcommand{\M}{{\mathcal M}}

\newcommand{\Spec}{\operatorname{Spec}}

\newcommand{\Hom}{{\operatorname{Hom}}}

\newcommand{\codim}{\operatorname{codim}}

\newcommand{\isom}{\simeq}

\renewcommand{\th}{^\mathrm{th}}
\newcommand{\chr}{\operatorname{char}}
\newcommand{\rank}{\operatorname{rank}}
\newcommand{\m}{\mathfrak{m}}

\title{Irreducibility for log arc schemes}
\author{Balin Fleming}
\email{bkflem@umich.edu}
\begin{document}
\begin{abstract}We characterise when the log arc scheme of a fine log scheme $(X, \M)$, with $X$ a variety over a field of characteristic zero, is irreducible. This generalises the theorem of Kolchin that the (ordinary) arc scheme of $X$ is irreducible exactly when $X$ is irreducible. \end{abstract}
\maketitle

\setcounter{section}{1}
\thispagestyle{empty}
\subsection{Introduction.} The recent construction of jet schemes in log geometry in general \cite{Dutter} leaves us with many natural questions about how much of the theory of ordinary jet schemes may be carried over to analogues in the setting of log geometry. Some progress has been made in developing this correspondence, for example Karu and Staal \cite{KS} give a log geometry version of Musta\c{t}\v{a}'s well-known theorem \cite{M} characterising when the ordinary jet schemes of a local complete intersection are irreducible. 

A log scheme is simply a scheme $X$ together with a sheaf of monoids $\M$ on $X$ and a morphism $\M \to \O_X$ to the multiplicative monoid of the structure sheaf $\O_X$ which induces an isomorphism on units. In the basic example, $\M$ is locally generated as a monoid by local equations for a chosen normal crossing divisor $D$. Sometimes one thinks of $D$ as a ``boundary'' or ``divisor at infinity'' for the space $X-D \subseteq X$. The formalism of the monoid sheaf $\M$, introduced in Kato's paper \cite{Kato}, gives a unified way to generate geometric data with logarithmic (that is, order one) poles along $D$. Although Kato pointed toward arithmetic uses, for example taking $X$ to be a scheme with semi-stable reduction over a discrete valuation ring and $D$ its closed fibre to construct crystalline cohomology with logarithmic poles, we are interested here in log schemes (over a field) as geometric objects in themselves.

In this paper we consider the question of irreducibility for log arc schemes. For ordinary arc schemes the matter was settled in characteristic zero as a consequence of a more general theorem of Kolchin in differential algebra: the arc scheme $J_\infty(X)$ of an irreducible variety $X$ is irreducible. This behaviour is much simpler than that of jet schemes, where the singular locus of the base scheme $X$ frequently gives rise to ``extra'' irreducible components in the jet scheme $J_m(X)$. Given a fine log scheme $(X, \M)$ there is an additional elementary way to find ``extra'' components in the log jet scheme $J_m(X, \M)$: if the rank of the log structure $\M$ is too large along some stratum of $X$ compared to the codimension of the stratum then the log jet scheme $J_m(X, \M)$ is thickened by an affine bundle along that stratum. (In the case where the log structure of $X$ comes from a normal crossing divisor $D$, the rank $j$ strata of $X$ are the intersections of exactly $j$ irreducible components of $D$.) There is no obstruction to lifting the bundle to higher order, so this carries over to the log arc scheme $J_\infty(X, \M)$, which then has the same ``extra'' component. What we show is that avoiding this situation is sufficient: when the log structure $\M$ has the expected rank everywhere for its stratification, if $X$ is irreducible then so is $J_\infty(X, \M)$.

\begin{theorem}\label{maintheorem} Let $k$ be a field of characteristic zero. Let $(X, \M)$ be a fine log scheme over $(\Spec k, k^*)$, with $X$ irreducible of finite type, and let $J_\infty(X, \M)$ be its log arc scheme. Let $X_j$ be the rank $j$ stratum of $(X, \M)$ and let $r$ be the minimum rank of $\M$ on $X$. Then $J_\infty(X, \M)$ is irreducible if and only if $\codim_X X_j = j - r$ for all non-empty $X_j$ (that is, $(X, \M)$ is dimensionally regular in the sense of Definition~\ref{dimreg}). \end{theorem}

Our exposition is essentially self-contained with respect to the content of the theory of log geometry that we will need. In Sections 2 and 3 we recapitulate some definitions and basic notions of commutative monoids and then of log schemes. Here and after we give a number of simple examples to illustrate how these notions may behave. In Section 4 we briefly recall the ordinary notions of jets and arcs, which underlie log jets and log arcs, and give two characterisations of log jet schemes: one functorial, corresponding to the view of the $k$-points of the jet or arc schemes of $X$ being the $k[t]/(t^{m+1})$-points or $k\llbracket t \rrbracket$-points of $X$, and one algebraic, corresponding to the view of (higher-order) differentials as being co-ordinates on jet or arc schemes. \'Etale maps between varieties work especially well for the induced maps on their jet or arc spaces, and in Section 5 we recall the notion of log \'etaleness to this end for maps of log jet or log arc spaces. In Section 6 we consider concretely the case of the log arc scheme of a monoid algebra. This is not strictly necessary for what follows, but the argument is not far from the definitions of Section 4, and the result here has fewer conditions on the log structure or the ground field $k$ than does Theorem~\ref{maintheorem}. 

Finally in Section 7 we introduce first our condition we will call {\em dimensional regularity} and put it in context in log geometry. It is automatic for example when a fine saturated log scheme $(X, \M)$ is log smooth, but is much weaker than this; it is just a combinatorial condition on the the stratification of $X$ by the rank of $\M$. We then prove our main result, Theorem~\ref{maintheorem}. The plan is to show that the closure of the space of log arcs at one level of the stratification includes the log arcs on the smooth locus of the next level down, and then use Kolchin's theorem for ordinary arc schemes to see that the log arcs on the smooth locus of each stratum are dense in the log arc space of the whole stratum.

\begin{remark} We stated Kolchin's theorem and our Theorem~\ref{maintheorem} for an irreducible scheme $X$. Another version of these is the statement that the irreducible components of the arc space $J_\infty(X)$ or log arc space $J_\infty(X, \M)$ biject with the irreducible components of $X$ through the projection maps $\pi: J_\infty(X) \to X$ or $\pi: J_\infty(X, \M) \to X$, assuming each component of $X$ is separately dimensionally regular. This statement with $X$ allowed reducible implies a fortiori Theorem~\ref{maintheorem}, but the two are in fact equivalent (and we will pass between them as convenient). The reason for this is that an arc, ordinary or log, on $X$ is contained inside some irreducible component $Z$ of $X$. That is to say, a map $\Spec k\llbracket t \rrbracket \to X$ factors through an inclusion $i: Z \to X$, because $\Spec k\llbracket t \rrbracket$ is integral. It follows that the preimage $\pi^{-1} Z$ of $Z$ is simply $J_\infty(Z)$ or $J_\infty(Z, i^*\M)$. (This does not occur with jets in place of arcs.) 

The same observation shows that an arc, ordinary or log, on $X$ factors through the reduced structure $X_{red}$ of $X$, so that it makes no difference whether we assume $X$ to be reduced or not. (Again for jets it certainly makes a difference.) \end{remark}

\begin{remark} Kolchin's theorem is known to hold exactly for perfect fields \cite{NS}, and one may ask for a criterion for irreducibility for log arc spaces in this generality as well. In Section~\ref{positivechar} we briefly consider how much of our argument for Theorem~\ref{maintheorem} carries over to perfect fields of positive characteristic. There the condition of dimensional regularity is not sufficient (see Example~\ref{positivecharexample}), so something more must be required. We do not attempt to give a necessary and sufficient condition in this case.

Of course, since every ordinary arc scheme is also a log arc scheme (for the trivial log structure), the counterexamples to Kolchin's theorem over non-perfect fields apply also to the irreducibility theorem for log arc schemes. \end{remark}

\subsection{Terminology.} In this note we are often concerned with terms and concepts (jets and arcs, smoothness and \'etaleness, and so forth) which appear in both the usual setting of the category of schemes and in the setting of the category of log schemes. Wherever a term is used in the log scheme sense we will indicate this by including ``log'' in its name (so, log jets and log arcs, log smoothness and log \'etaleness, and so forth). Sometimes when a term is used in its usual scheme-theoretic sense we will emphasise the distinction by calling it ``ordinary'' (so, ordinary jets and ordinary arcs, and so forth). 

\subsection{Acknowledgements.} The author thanks Kalle Karu for his generous advice and discussion, as well as for helpful comments on a draft of this paper. 

This work was supported by a fellowship from the Department of Mathematics at the University of Michigan, while visiting Prof. Karu and the Department of Mathematics at the University of British Columbia.

\section{Monoids}

By a {\em monoid} we mean always a commutative semigroup with an identity element. Most often we will write the monoid operation multiplicatively, but sometimes additive notation will be apt. In particular we will use the natural numbers $(\N, +)$ under addition, where $\N = \{0, 1, 2,...\}$, to stand for the free monoid on one generator. When the generator is specified to be some element $x$ of a ring $R$, we write by abuse of notation $\N x = \{ 1, x, x^2, ...\}$ to stand for the set of powers of $x$. Since we never consider the set of integer multiples of $x$ we hope that no confusion arises thereby.

\subsection{Monoid algebras.} To a multiplicative monoid $P$ we associate the monoid algebra $k[P]$, which by definition is the quotient of the polynomial ring on the set $P$ by the relations between monomials in the variables which hold of them in the monoid $P$ (and by identifying the neutral elements $1 \in P$ and $1 \in k$). Equivalently, one may take construct $k[P]$ as the quotient of the polynomial ring on a set of generators of the monoid $P$ by the relations which hold among the monomials in these generators. Thus $k[P]$ is a quotient of a polynomial ring by a {\em pure binomial ideal}; we also sometimes say that it is a quotient given by {\em monomial relations}. Conversely such a presentation $$R = k[\{x_i\}_{i \in I}]/J$$ with a pure binomial ideal $J$ determines a monoid $P$, generated by the symbols $x_i$, for which $R = k[P]$.

By abuse of terminology we also call the affine scheme $\Spec k[P]$ a monoid algebra, (partly to avoid conflation with the notion of ``monoidal space'' which has been used by other authors to refer to something different).

\subsection{Basic notions and properties.} A map $P \to Q$ of monoids induces a $k$-algebra map $k[P] \to k[Q]$.

A monoid $P$ is {\em finitely generated} if there is a surjection $\N^r \to P$ for some integer $r \geq 0$. Equivalently, all the elements of $P$ may be written as monomials in some finite generating set of elements of $P$. Such a description gives $\Spec k[P]$ as a closed subscheme of affine space $\A^r = \Spec k[\N^r]$. 

A monoid $P$ has a {\em group completion} $P^{gp}$, that is, a group $P^{gp}$ with a map $P \to P^{gp}$ with the universal property that a monoid morphism $P \to G$ to any group $G$ factors through $P^{gp}$. One may construct $P^{gp}$ as the set of fractions $pq^{-1}$ with $p, q \in P$, where $pq^{-1} = rs^{-1}$ in $P^{gp}$ if there is $t \in P$ such that $tps = trq$. 

A monoid $P$ is called {\em integral} if whenever an equation $pq = pr$ holds in $P$ one has $q = r$. Equivalently, $P$ is integral if and only if the map $P \to P^{gp}$ is an inclusion. This gives also an inclusion $k[P] \to k[P^{gp}]$. If in addition $P^{gp}$ is torsion-free then $k[P^{gp}]$ is a domain, according to a classical result for group rings, so that $k[P]$ is then also a domain. (But $P$ being integral, or integral and torsion-free, is not sufficient for $k[P]$ being a domain in general.)

A monoid both integral and finitely generated is called {\em fine}.

An integral monoid $P$ is called {\em saturated} if whenever some power $g^n$ lies in $P$, where $g \in P^{gp}$ and $n \geq 1$, one has $g \in P$ also. 

%Commonly in log geometry one is concerned with monoids which are fine and saturated, and this is often abbreviated to saying that the monoid is {\em fs}.

To a monoid $P$ there is associated a universal integral monoid $P^{int}$, which may be realised as the image of $P \to P^{gp}$. This in turn has a universal saturation $P^{sat} \subseteq P^{gp}$, consisting of all the elements of $P^{gp}$ for which some positive power lies in $P^{int}$.

The set $P^*$ of invertible elements of $P$ is its group of units. A monoid $P$ with $P^* = 1$ is called {\em sharp}. A monoid $P$ has a universal map to a sharp monoid, namely to the quotient $P/P^*$ (see Section~\ref{idealsfacesquotients} below).

\begin{example} A rational cone in an integer lattice is a fine, saturated, torsion-free monoid. Conversely, if $P$ is a fine saturated torsion-free monoid, then $P^{gp}$ is torsion-free, for the saturation of any submonoid of $P^{gp}$ includes all the torsion elements of $P^{gp}$. Then $P^{gp}$ is free and $P$ is a rational cone in $P^{gp}$. The monoid algebra $\Spec k[P]$ is a (normal) affine toric variety. \end{example}

\begin{example}\label{cuspidalcubic} The subset $P = \N-\{1\} = \{0, 2, 3, 4,....\}$ of the natural numbers under addition is a monoid. Writing this multiplicatively as powers of a variable $t$, its monoid algebra is $$k[P] = k[t^2, t^3] \isom k[x, y]/(x^3 - y^2),$$ the co-ordinate ring of a cuspidal plane cubic curve. The inclusion $P \subseteq \N$ gives $P^{gp} \subseteq \N^{gp} = \Z$. In fact $P^{gp} = \Z$, for example because $1 = 4 - 3$ is a difference of elements of $P$ (in other words, because $P^{gp}$ is a subgroup of $\Z$ containing $2, 3$). We see that $P$ is integral and finitely generated, so is fine, but is not saturated. Its saturation is $\N$, and the natural map $k[P] \to k[P^{sat}] = k[t]$ is the normalisation map of the cuspidal curve. \end{example}

\begin{example}\label{grouphastorsion} Let $Q$ be generated by two elements $x, y$ subject to the single relation $x^2 = y^2$. Then $Q$ is integral and torsion-free. The spectrum of the monoid algebra $k[Q] = k[x,y]/(x^2 - y^2)$ consists of two lines meeting at the point $x = y = 0$. One has $$Q^{gp} \isom (\Z/2\Z) \omega \times \Z x$$ by identifying the generator $y$ with $\omega\cdot x$, and $$Q^{sat} = (\Z/2\Z) \omega \times \N x \subseteq Q^{gp}.$$ So $k[Q^{sat}] = k[x, \omega]/(\omega^2 -1).$ The map $\Spec k[Q^{sat}] \to \Spec k[Q]$ glues the two lines of $\Spec k[Q^{sat}]$ with $\omega = \pm 1$ together at $x = 0$, while $\Spec k[Q^{gp}] \to \Spec k[Q]$ just includes the complement of the node of $\Spec k[Q]$. \end{example}

\subsection{Ideals, faces, and quotients.}\label{idealsfacesquotients} An {\em ideal} $I$ of $P$ is a submonoid of $P$ closed under multiplication by elements of the monoid, $IP \subseteq I$. An ideal $I$ is called prime if whenever $pq \in I$ one of $p, q$ lies in $I$. Equivalently, an ideal $I$ is prime if the complementary set $P - I$ is a submonoid of $P$ (in fact, a face of $P$: see below). 

For an ideal $I$ of $P$ there is a quotient monoid $P \to P/I$ which may be realised as the set $P - I$ together with a ``zero element,'' corresponding to the subset $I$. If a product of elements in $P-I$ lies in $I$ then in $P/I$ the product is zero (i.e., it is the class of $I$). If $I$ is prime then this does not happen, and $P/I$ may be realised as the {\em monoid} $P-I$ together with a zero element, (where zero times $P-I$ is zero).

A {\em face} $F$ of a monoid $P$ is a submonoid of $P$ such that whenever $p, q\in P$ have $pq \in F$, then in fact both $p, q$ lie in $F$. The complementary set $P - F$ of a face $F$ is a prime ideal of $P$, and vice versa. The unit group $P^*$ of $P$ is a face of $P$, in fact the minimum face of $P$. %For a face $F$ of $P$ there is a monoid $F^{-1}P$ obtained by inverting the elements of $F$.

There is a natural quotient $P/P^*$, realised as the set of cosets of $P^*$ in $P$. This is a different kind of quotient of $P$ than those by ideals of $P$. %For a face $F$, there is a quotient of $P$ by $F$, by which we mean the map $P \to F^{-1}P/(F^{-1}P)^*$ obtained by inverting $F$ and taking the quotient by the unit group of $F^{-1}P$. The unit group of $F^{-1} P$ is generated as a group by $P^*$ and $F$.

\begin{example} The usual group quotient $\Z/2\Z$ is a monoid of two elements, as is the monoid quotient $\N/\{1, 2, ...\}$, but they are not the same monoid: the latter has a non-trivial idempotent $1 + 1 = 1$. One might write it instead as the set $\{0, \infty\}$, with $\infty$ being the additive version of what the zero element is for multiplicative monoids.\end{example}

\begin{example} View the monoid $\N^2$ as the integer lattice points of the first quadrant of the plane. It has three proper faces, namely the origin and the two copies of $\N$ along the co-ordinate axes, so three prime ideals, namely the complement of the origin and the points of positive vertical or horizontal co-ordinate. Any point $p \in \N^2$ generates an ideal $p + \N^2$, a translated cone in the plane. The ideals of $\N^2$ are unions of such cones (which may then be realised as finite unions of such cones). Equivalently, an ideal of $\N^2$ is the set of points above a descending envelope in the first quadrant. \end{example}

\begin{example} A field $k$, considered as a multiplicative monoid, has unit group $k^*$, and the quotient monoid $k/k^*$ is the (sharp) two-element multiplicative monoid $\{0, 1\}$. Here there is an asymmetry between the two points of $k/k^*$ in that the unit group acts freely on the orbit $k^* \subseteq k$ of $1 \in k$ but trivially on the orbit $\{0 \} \subseteq k$. \end{example}

\section{Log schemes}

A log scheme will be a scheme $X$ together with some combinatorial data encoded in a sheaf of monoids $\M$ on $X$. Any scheme may become a log scheme in many different ways, including trivially, so in complete generality $\M$ does not provide much control. There are various conditions on the pair $(X, \M)$ one may ask for to provide a measure. For example one commonly works in the category of log schemes with fine and saturated log structure, which avoids many possible pathologies in both $\M$ and $X$. Fine and saturated log smooth varieties are modelled in a strong sense on toric varieties \cite{KatoTS}. Like toric varieties they are normal and Cohen-Macaulay. But the general local model for a log scheme is just of an arbitrary map $X \to \Spec k[P]$ of a scheme $X$ to a monoid algebra.

\subsection{Log structures and charts.} A log scheme over $k$ is a pair $(X, \M_X)$, where $X$ is a scheme over $k$ and $\M_X$ is a sheaf of monoids on $X$ with a monoid morphism $\alpha_X: \M_X \to \O_X$ to the multiplicative monoid $(\O_X, \cdot)$ that induces an isomorphism $\M_X^* = \alpha_X^{-1} \O_X^* \to \O_X^*$ on units. The sheaf $\M_X$ is called a {\em log structure} on $X$. The map $\alpha_X$ need not be injective, although in many examples $\M_X$ will be a subsheaf of $\O_X$. The category of log structures on $X$ has an initial object, which is $\O_X^*$ (as a sub-monoid sheaf of $\O_X$), also called the trivial log structure on $X$, and a final object, which is $\O_X$, since by definition there are natural maps $\O_X^* \to \M_X \to \O_X$.

Any sheaf of monoids $P_X$ on $X$ with a monoid morphism $\alpha: P_X \to \O_X$ (this much data is called a {\em pre-log structure}) generates an associated log structure $P_X^a$, which may be realised as the fibred sum (in the category of sheaves of monoids) of the diagram \begin{center} $ \xymatrix{ & P_X \\ \O_X^* & \alpha^{-1} \O_X^* \ar[l] \ar[u]\\ } $ \end{center}

In particular if one specifies a monoid $P$ and a map $P \to \O_X(X)$ one obtains a log structure on $X$ generated by $P$, namely, that associated to the constant sheaf $P$ determines. A monoid $P$ is called a {\em chart} for the log structure it generates. 

A log structure on $X$ is called {\em coherent} if it has, Zariski-locally, charts by finitely generated monoids. Likewise the log scheme $(X, \M_X)$ is called fine or saturated if it has, Zariski-locally, charts by fine or saturated monoids. A map $P \to \M_X(U)$ over an open set $U \subseteq X$ is a chart if and only if it induces isomorphisms $P^a_x \to \M_{X, x}$ at stalks for $x \in U$. In any case, a map $P \to \O_X(U)$ determines a map $U \to \Spec k[P]$ to the monoid algebra of $P$. Sometimes this map to $\Spec k[P]$, rather than from the monoid $P$, is called a chart on $U$. 

\begin{example} A monoid algebra $k[P]$ naturally gives a log scheme, with the log structure given by the chart $P$. We call this the standard structure of $\Spec k[P]$ as a log scheme. \end{example}

\begin{example} One of the principal motivating examples arises from a scheme $X$ together with a (Cartier) normal crossing divisor $D$ on $X$. For a point $x \in X$ let $D_1, ..., D_r$ be the prime components of $D$ passing through $x$ and take a chart $\N^r$ near $x$, with the standard generators of $\N^r$ mapping to local equations for the components $D_i$. Since the choice of equations is not canonical these charts typically will not glue together by themselves, hence one expands them as monoids to include the units near $x$. Now various constructions one makes from the log structure, like the sheaf of log differentials, are interpreted as objects with logarithmic (i.e., order one) poles along $D$. The dual objects, like the sheaf of log tangent vectors (and, we shall see, log jets in general) are thereby interpreted as objects with zeroes of order one along $D$.

More generally, let $U$ be an open set in $X$, and take $\M$ to be the subsheaf of $\O_X$ of functions invertible on $U$. That is, for $V \subseteq X$ open put $$\M(V) = \{ f\in \O_X(V) \text{ such that } f\vert_U \in \O_X(U \cap V)^*\}.$$ This is a monoid sheaf under multiplication, and is a log structure on $X$. When $X - U$ is a normal crossing divisor $D$, this $\M$ is the same log structure as constructed above. In general, this $\M$ is the pushforward (see Section~\ref{catlogschemes} below) of the trivial structure on $U$ along the inclusion $U \to X$. \end{example}

\begin{example}\label{a2ex} The affine plane $X = \Spec k[\N^2] = \Spec k[x,y]$ with its standard log structure $\M$, generated as a submonoid sheaf of $\O_X$ by the monomials $x$ and $y$ together with $\O_X^*$, is a special case of the last two examples. 

One may go from the chart $P = \N^2$ to the associated log structure $\M$ as follows. On the open subset $U = \Spec k[x,y]_x$ of $X$ the monomial $x$ is a unit, so that in the fibred sum of $P$ and $\O_X(U)^*$ the element $x \in P$ is identified with the unit $x \in \O_X(U)^*$. Consequently on $U$ the monoid $\M$ is generated just by the monomial $y$ together with $\O_X^*$. Likewise on the open set $V = \Spec k[x,y]_y$ the log structure $\M$ is generated by $x$. On $U \cap V = \Spec k[x,y]_{xy}$ the log structure is trivial, $\M\vert_{U \cap V} = \O_{U\cap V}^*$. \end{example}

\begin{example}\label{a1twocharts} Consider the affine line $\Spec k[x]$ with log structure $\M$ given by a chart $\N \to k[x]$ taking the generator of $\N$ to the polynomial $x(x-1)$. That is, the chart is given by the inclusion of algebras $k[x(x-1)] \to k[x]$. The global sections of $\M$ are not just the monoid sum $k[x]^* \oplus \N$. In fact, thinking of $\M$ as a subsheaf of $\O_X$, both $x$ and $x-1$ are global sections. For example, away from $x = 0$ the function $x$ is a unit, so is a section of $\M$, while away from $x = 1$ the function $$x = (x-1)^{-1} \cdot x(x-1)$$ is a unit $(x-1)^{-1}$ times a section $x(x-1)$, so is a section of $\M$. So $x$ is a global section; and similar for $x-1$. 

In this example one has two global charts, by $\N x(x-1)$ and by $\N^2 = \N x \oplus \N (x-1)$, whose (abstract) group completions have different rank. The latter is closer to the global structure of $\M$, in the sense that $\M(X) \isom \N^2 \oplus k[x]^*$, while the former is closer to the local structure of $\M$ at the points $x = 0, 1$, in the sense that the induced map $$\N x(x-1) \to \M_p/\O_{X,p}^*$$ is an isomorphism at these two points $p$.\end{example}

\begin{remark}\label{goodchartdef} A chart $P \to \M_X$ of a fine log scheme $(X, \M_X)$ which induces an isomorphism $P \to \M_{X, x}/\O_{X,x}^*$ is called {\em good} at $x$. Good charts do not always exist. The next proposition gives some simple cases. \end{remark}

\begin{prop}\label{goodchart} Let $(X, \M)$ be a fine log scheme, $x$ a point of $X$, and $P = \M_x/\O_{X, x}^*$. Then: \begin{itemize} \item[(1)] If $P^{gp}$ is torsion-free there is a chart $P \to \M$ near $x$. \item[(2)] If $X$ is normal near $x$ then there is a chart $P \to \M$ near $x$. \end{itemize} \end{prop}

\begin{proof} By (\cite{Ogus}, II.2.3.6) there is an isomorphism $$\M(U)/\O_X(U)^* \to \M_x/\O_{X, x}^*$$ in some neighbourhood $U$ of $x$. So $P$ would be made into a chart by a section of the monoid morphism $\M(U) \to \M(U)/\O_X(U)^*$. To study this map we consider the induced map on group completions, \begin{center} $ \xymatrix{ \M(U) \ar[r] \ar[d] & \M(U)/\O_X(U)^* \ar[d] \\ \M(U)^{gp} \ar[r] & (\M(U)/\O_X(U)^*)^{gp}.\\ } $ \end{center} Since $\M(U)$ and $\M(U)/\O_X(U)^*$ are integral monoids, the vertical arrows are injections, and a splitting of the bottom arrow splits the top arrow by restriction.

Now in case (1) it is assumed that $P^{gp} = (\M(U)/\O_X(U)^*)^{gp}$ is free abelian, so a splitting of the bottom arrow exists in this case.

Otherwise, $P^{gp}$ has some torsion part, and the problem becomes to determine that torsion elements of $P^{gp}$ come from torsion elements of $\M(U)^{gp}$. Let $fg^{-1} \in P^{gp}$ be a torsion element, with $f, g \in \M(U)$ and $f^n = ug^n$ for some $u \in \O_X(U)^*$. If $X$ is normal on $U$, then $u = (\alpha(f) / \alpha(g))^n$ has an $n\th$ root $v = \alpha(f)/\alpha(g)$ in $\O_X(U)$. Then $f^n = (vg)^n$ in $\M(U)$, so $f(vg)^{-1}$ is torsion in $\M(U)^{gp}$. Choosing a decomposition of the torsion part of $P^{gp}$ as a sum of cyclic groups and lifting generators $fg^{-1}$ for each now gives (2). \end{proof}

\begin{remark} The proof also shows that, when $X$ is a fine log scheme over a field of characteristic zero, $\M_x/\O_{X,x}^*$ is ``a chart for $(X, \M)$ in an \'etale neighbourhood of $x$,'' (which notion we have not defined here), given by extracting the roots $v = u^{1/n}$ near $x$. In characteristic $p > 0$ the same happens if the torsion order of $\M_x^{gp}/\O_{X,x}^*$ is prime to $p$.\end{remark}

\begin{example} Let $X = \Spec k[x, y, z, w]/(x^2 z - y^2 w)$ with its standard structure as a monoid algebra. On the open subset $zw \neq 0$, there are not Zariski-local good charts along the locus $x = 0, y = 0$. The obstruction is that $z/w = (x/y)^2$ is a unit here but $x/y$ is not a regular function. \end{example}

\subsection{The category of log schemes.}\label{catlogschemes} A morphism of log schemes $(X, \M_X) \to (Y, \M_Y)$ is a morphism of schemes $f: X\to Y$ with a morphism of sheaves of monoids $\M_Y \to f_* \M_X$ compatible with $\O_Y \to f_* \O_X$; that is, making a commutative diagram \begin{center} $ \xymatrix{ \M_Y \ar[r]\ar[d] & f_* \M_X\ar[d] \\ \O_Y \ar[r] & f_* \O_X\\ } $ \end{center}

Given a map of schemes $f: X \to Y$ and a (pre-)log structure $\M_Y$ on $Y$, there is a pullback log structure $\M_X = f^* \M_Y$ on $X$, which is the log structure associated to the set-theoretic pullback $f^{-1} \M_Y$ with the monoid map $f^{-1} \M_Y \to f^{-1} \O_Y \to \O_X$. Given instead a (pre-)log structure $\M_X$ on $X$ there is a pushforward log structure $\M_Y$ on $Y$. Writing still $f_*$ for the usual set-theoretic pushforward of sheaves, this is the log structure associated to the fibre product (in the category of sheaves of monoids on $Y$) of the diagram \begin{center} $ \xymatrix{ & f_* \M_X\ar[d] \\ \O_Y^* \ar[r] & f_* \O_X^* \\ } $ \end{center} 

In either case, there is an induced map of log schemes $(X, \M_X) \to (Y, \M_Y)$. Pullback and pushforward are functorial, and they are adjoint functors in the usual way. Formation of the associated log structure from a chart or other pre-log structure commutes with pullback. 

A map of log schemes $f: (X, \M_X) \to (Y, \M_Y)$ such that $\M_X \isom f^* \M_Y$ is called {\em strict}. A monoid $P$ mapping to $\O_Y$ is a chart on $Y$ if and only if the map $Y \to \Spec k[P]$ is strict. Since strictness is preserved by composition, if $(X, \M_X) \to (Y, \M_Y)$ is strict then $P$ pulls back to a chart on $X$ by $X \to Y \to \Spec k[P]$.

If $(X, \M_X)$ and $(Y, \M_Y)$ are log schemes over a base $(S, \M_S)$, they have a fibre product log scheme $$(X \times_S Y, (\M_X \oplus_{\M_S} \M_Y)^a)$$ with underlying ordinary scheme $X \times_S Y$ and log structure generated by the fibred sum of the diagram \begin{center} $ \xymatrix{ & \M_X \\ \M_Y & \M_S \ar[u] \ar[l] \\ } $ \end{center} of sheaves of monoids on $X \times_S Y$, where we have written just $\M_X,$ etc. for the pullbacks of $\M_X,$ etc. to $X \times_S Y$.

\subsection{Stratification of fine log schemes}\label{strataintro} (\cite{Ogus}, II.2.3). Let $(X, \M)$ be a fine log scheme. At any point $x$ of $X$ there is a stalk $\M_x$ of the log structure sheaf, which is a monoid in the natural way. Consider the group $$\M_x^{gp}/\O_{X, x}^* = (\M_x/\O_{X, x}^*)^{gp}.$$ This is a finitely-generated abelian group, as one may see by taking a finitely-generated chart of $\M$ near $x$. Its rank we call the {\em rank of} $\M$ at $x$. 

The rank of $\M$ is an upper-semicontinuous function on $X$. Consequently there is a (finite) stratification of $X$ by locally closed subsets $X_j$ on which $\M$ has rank $j$. In this stratification, if $Z$ is a component of $X_j$ and $Z'$ is a component of $X_k$ such that the closure $\overline{Z}$ meets $Z'$, then $Z'$ is contained in $\overline{Z}$, and $j \leq k$. 

In general, when $X$ is irreducible still $X_0$ need not be dense in $X$, since the log structure $\M$ may have elements mapping to zero in $\O_X$. Instead there is some minimum rank $r$ of $\M$ on $X$, and $X_r$ is open and dense in $X$. The complement $$X - X_r = \bigcup_{j \geq r+1} X_j$$ is a Cartier divisor on $X$, which we might call  the {\em locus} $Z(\M)$ of the log structure on $X$.

\begin{example} In the notation of Example~\ref{a2ex}, the rank zero stratum of the affine plane $X$ with its standard structure is the complement $U \cap V$ of the co-ordinate axes. On $U - V$ the rank of $\M$ is one, since the quotients $\M_p^{gp}/\O_{X,p}^*$ at each point $p \in U - V$ are copies of $\Z$, generated by the monomial $y$. Likewise on $V - U$ the rank of $\M$ is one. These punctured lines together form the rank one stratum of $(X, \M)$. At the origin $o$, which is the single point of $X - (U \cup V)$, one has $\M_o^{gp}/\O_{X,o}^* \isom \Z^2$, generated by the image of the standard chart $\N^2$. So the origin is the rank two stratum of $(X, \M)$. \end{example}

\begin{example}\label{a2nstdex} We might instead give the plane $X = \Spec k[x,y]$ a non-standard log structure $\M' \subseteq \O_X$ generated by the monomial $xy$. That is, the log structure is given by the chart $k[xy] \to k[x,y]$. The difference from the standard structure $\M$ is at the origin, where now the log structure $\M'$ has rank one. So it is the two co-ordinate axes together, including the origin of the plane, which are the rank one locus of $(X, \M')$. In particular this stratum is singular as a subscheme of $X$. \end{example}

\begin{example}\label{quadrconefibr} Consider the affine cone $X = \Spec k[x, y, z, w]/(xw-yz)$ over the quadric surface, with (non-standard) chart given by the projection $$f: \Spec k[x, y, z, w]/(xw-yz) \to \Spec k[x,y]$$ to the plane. The  rank zero stratum of $X$ is the inverse image of the rank zero stratum $xy \neq 0$ of the plane, and the inverse image of the rank one stratum $(x = 0, y \neq 0) \cup (x \neq 0, y = 0)$ of the plane is a line bundle consisting of two components $Z, Z'$ of the rank one stratum $X_1$ of $X$. Over the origin $(x, y) = (0, 0)$ the fibre $f^{-1}(0)$ of $f$ is the plane $\Spec k[z, w]$. The lines $zw = 0$, which are the closure of the components $Z, Z'$, are the rank two stratum $X_2$. The complement of $X_2$ in $f^{-1}(0)$ is another component $Z''$ of $X_1$, for on the locus $zw \neq 0$ in $X$ the monomials $x$ and $y$ are related by units. \end{example}

\begin{remark}\label{centralpoint} A point $x$ of a fine log scheme $(X, \M)$ which is in the closure of every component of every stratum is called a {\em central point} for the log scheme. When there is such, the stratum to which $x$ belongs consists of central points. If $P$ is sharp then $k[P]$ has a central point: if it is given as $\Spec k[\{x_i\}_{i \in I}]/J$ for some variables $x_i$, then the origin (where all $x_i = 0$) is central. 

A point $x$ always has a neighbourhood of which it is a central point, namely one obtained by deleting the closures $\bar{Z}$ of components of strata such that $x \not\in \bar{Z}$. While we do not need this concept explicitly in our exposition, it simplifies the local picture of the stratification and underlies some basic local results. For example, the restriction map $\M(U)/\O(U)^* \to \M_x/\O_{X, x}^*$ is an isomorphism if $x$ is a central point of $U$, as we recalled in the proof of Proposition~\ref{goodchart}. \end{remark}

\section{Log jets and log differentials}

We begin with the functorial characterisation of log jet schemes and follow with their concrete realisation in terms of log Hasse-Schmidt differentials.

\subsection{Ordinary jets and arcs.} Let $$j_m = \Spec k\llbracket t \rrbracket/(t^{m+1}),$$ for $m \geq 0$. An ordinary $k$-valued $m$-jet on $X$ is a map $j_m \to X$, and an ordinary $S$-valued $m$-jet on $X$ is a map $S \times j_m \to X$. The functor $$S \mapsto \Hom(S \times_k j_m, X)$$ is representable, and we write $J_m(X)$ for the (ordinary) {\em jet scheme} that represents it. There are natural maps $J_m(X) \to J_n(X)$ for $m \geq n$ induced by the truncation $k\llbracket t \rrbracket/(t^{m+1}) \to k\llbracket t \rrbracket/(t^{n+1})$. 

An ordinary $k$-valued arc on $X$ is a map $$j_\infty = \Spec k \llbracket t \rrbracket \to X$$ and an ordinary $S$-valued arc on $X$ is a map $S \times_k j_\infty \to X.$ There is again a scheme of such, denoted $J_\infty(X)$, and it is the projective limit of the spaces $J_m(X)$ with their truncation maps. 

\subsection{Log jets and arcs.} Now let us give $j_m$ its trivial log structure, which as a sheaf is just a copy $$k\llbracket t\rrbracket /(t^{m+1})^* = \{ a_0 + a_1 t + \cdots a_m t^m \text{ mod } t^{m+1} \text{ such that } a_0 \neq 0\}$$ of the units in $k\llbracket t\rrbracket /(t^{m+1})$ at the sole point of $j_m$. Let $S$ be a scheme, made into a log scheme through its final log structure $\O_S$. Recall from Section~\ref{catlogschemes} that there is a product log scheme of $S$ and $j_m$ in the category of log schemes over $(\Spec k, k^*)$, whose underlying ordinary scheme is $S \times_k j_m$ and whose log structure is the pushout $$\O_S \oplus_{k^*} \O_{j_m}^* \to \O_{S \times_k j_m} \isom \O_S[t]/(t^{m+1})$$ of the respective log structures fibred over $k^*.$ When $S = \Spec k$, this log structure may be described, by shifting multiplicative constants to the left factor, as the (non-fibred) product $$k \times (1 + t k\llbracket t \rrbracket/(t^{m+1}))$$ of $k$ and the principal units of $k\llbracket t \rrbracket /(t^{m+1})$, with the map to the structure sheaf of $j_m$ being just the multiplication map $$k \times (1 + t k \llbracket t \rrbracket /(t^{m+1})) \to k\llbracket t \rrbracket / (t^{m+1}).$$ We note that the image of this multiplication map consists only of the units of $k\llbracket t \rrbracket / (t^{m+1})$ together with zero. There is a similar description of the product log structure on $S \times_k j_m$ when $S$ is any affine scheme. 

In general, an $S$-valued log $m$-jet on $(X, \M_X)$ is a map of log schemes $$(S \times_k j_m, \O_S \oplus_{k^*} \O_{j_m}^*) \to (X, \M_X).$$ The functor $$S \mapsto \Hom_{log}(S\times_k j_m, X)$$ (where we have left the log structures out of the notation on the right side) is representable, and we write $J_m(X, \M_X)$ for the {\em log jet scheme} that represents it. Again there are truncation maps $J_m(X, \M_X) \to J_n(X, \M_X)$ where $m \geq n$. 

It is worth writing down explicitly the data of a $k$-valued log $m$-jet on a log scheme $(X,\M_X).$ Working locally near the image of $\Spec k\llbracket t \rrbracket/(t^{m+1}) \to X$, we may take $X = \Spec A$ with a chart $P$. A map of log schemes $$\Spec k \times_k j_m \to (X, \M_X)$$ is then equivalent to a commutative diagram 
\begin{center} $ \xymatrix{ k \times (1 + t k\llbracket t \rrbracket/(t^{m+1})) \ar[d]^{mult} & P \ar[l]\ar[d] \\ k\llbracket t \rrbracket/(t^{m+1}) & A. \ar[l]\\ } $ \end{center} Here the bottom row is the map of the underlying ordinary jet and the top row is the map on log structures.

Replacing everywhere $j_m$ by $j_\infty$, one has the definition of an $S$-valued log arc. The space of such is $J_\infty(X, \M_X)$, and it is the projective limit of the schemes $J_m(X, \M_X)$ with their truncation maps.

\begin{remark} Although when $S = \Spec k$ one has $S \times_k j_m = j_m$ as schemes, the log structure this space is given is not any of the ``obvious'' choices, and in particular is not integral and does not inject into $\O_{j_m}$. \end{remark}

\begin{remark} One might also ask after $(S, \M_S)$-valued log $m$-jets (or arcs), with different log structures $\M_S$ on $S$. These functors on the category of log schemes over $(\Spec k, k^*)$ are representable, by log schemes $(Y_m, \M_{Y_m})$. Here the underlying scheme $Y_m$ is just the scheme $J_m(X, \M_X)$ described above, considering only the structure $\O_S$ on $S$, because $S \mapsto (S, \O_S)$ is right adjoint to the forgetful functor from log schemes to schemes (because $(S, \O_S)$ is initial in the category of log schemes underlain by $S$). The log structure one obtains on $J_m(X, \M_X)$ from this construction in the category of log schemes is just that pulled back from the map $J_m(X, \M_X) \to X$. \end{remark}

\subsection{Log Hasse-Schmidt differentials.}\label{HSdiffintro} The existence of log jet schemes relative to arbitrary morphisms $(X, \M_X) \to (Y, \M_Y)$ was proven in \cite{Dutter} by construction in terms of log Hasse-Schmidt differentials, (although some cases, especially that of the log structure associated to a normal crossing divisor were considered prior to this, e.g. in \cite{N}, \cite{DL}). This construction was modelled on the construction in \cite{Vojta} of ordinary jet schemes in terms of ordinary Hasse-Schmidt differentials. For each $m \geq 0$ one constructs locally the $m\th$ (log) Hasse-Schmidt algebra, the spectrum of which is the scheme of (log) $m$-jets.

Locally, with $X = \Spec A$ is affine, one has the log Hasse-Schmidt algebra $HS^m_X(\M)$, relative to the base log scheme $(\Spec k, k^*)$, as follows. Start with the polynomial ring over $A$ on symbols $d_i f$ and $\partial_j p$, for every $f \in A, p \in \M(X)$, and $1 \leq i, j \leq m$. It is convenient to introduce notation $d_0 f = f$ for every $f \in A$ and $\partial_0 p = 1$ for every $p \in \M(X)$. Then take the quotient by the relations \begin{itemize} \item[(1)] $d_i (f + g) = d_i f + d_i g$ for $f, g \in A$,
\item[(2)] $d_i c = 0$ for $c \in k$ and $i \geq 1$,
\item[(3)] $d_i \alpha(p) = \alpha(p) \partial_i p$ for $p \in \M(X)$, and 
\item[(4)] the ordinary and ``logarithmic'' Leibniz rules $$d_k (fg) = \sum_{\stackrel{0 \leq i, j \leq k}{i + j = k}} d_i f d_j g$$ for $f, g \in A$ and $$\partial_k (pq) = \sum_{\stackrel{0\leq i, j \leq k}{i+j = k}} \partial_i p \partial_j q$$ for $p, q \in \M(X)$. \end{itemize}

Note that in the logarithmic Leibniz rule the terms $\partial_k p$ and $\partial_k q$ appear on the right side with co-efficients $1 = \partial_0 q = \partial_0 p$. In particular the case $k = 1$ asserts that $$\partial_1(pq) = \partial_1 p + \partial_1 q.$$ One thinks of $d_i$ as like a divided differential $\displaystyle \frac{1}{i!} d^i$, and $\partial_1$ as the logarithmic differential $d \log$, and $\partial_i$ for $i \geq 2$ not as repeated logarithmic differentiation but as the differential $\partial_i p = (d_i \alpha(p))/\alpha(p)$ with a logarithmic pole along the locus $\alpha(p) = 0$.

One sees, using the ordinary and ``logarthmic'' quotient rules, that this construction localises properly, so that the spectra of the locally constructed $m\th$ log Hasse-Schmidt algebras on some $(X, \M)$ glue to a global object, which is the space of log $m$-jets. In terms of the $k$-valued log jets, the connection between the points $HS^m_X(\M) \to k$ of the log Hasse-Schmidt algebra and morphisms $P \to k \times (1 + t k\llbracket t \rrbracket /(t^{m+1})$ is that in the latter the first component corresponds to values for the elements $\alpha(p) = d_0 \alpha(p)$, and the co-efficients of the series in the second component correspond to values for $\partial_i p$.

%\begin{remark}\label{truncmapsaffine} One sees easily from the description by Hasse-Schmidt algebras that the truncation maps $J_m(X, \M) \to J_n(X, \M)$ are affine morphisms: for locally on $X$ they just correspond to the maps $HS^n_X(\M) \to HS^m_X(\M)$ induced by identifying a symbol $d_i f$ or $\partial_j p$ in $HS^n_X(\M)$ with the same in $HS^m_X(\M)$. \end{remark}

\begin{remark}\label{logleibnizeqns} As above, one has $$\partial_1 \left(\prod p_j^{e_j}\right) = \sum e_j \partial_1 p_j.$$ For $k \geq 2$ one does not have the same identity, but may still write $$\partial_k \left(\prod p_j^{e_j} \right)= \sum e_j \partial_i p_j + G_k$$ for some universal polynomial $G_k$ in the lower-order differentials $\partial_i p_j$ with $i < k$.\end{remark}

\begin{remark} In general one may not replace $\M(X)$ by an arbitrary chart for $\M$ in the construction of $HS^m_X(\M)$. For example, in Example~\ref{a1twocharts} when $m = 1$ the chart $\N^2$ gives both log differentials $\partial_1 x$ and $\partial_1 (x-1)$, while the chart $\N$ gives only their sum $\partial_1 (x(x-1)) = \partial_1 x + \partial_1 (x-1)$. However, charts will do for computing the Hasse-Schmidt algebra locally on $X$. \end{remark}

\begin{remark} So far we have preferred multiplicative notation for our monoids, since in many cases we think of them just as sub-monoid sheaves of $(\O_X, \cdot)$. The log derivative $\partial = \partial_1$ gives an isomorphism from an integral monoid written multiplicatively to the same monoid written additively. Where we interpret the monoid written multiplicatively in terms of monomials, this interprets the monoid written additively as a submonoid of its space of first-order log differentials. \end{remark}

\subsection{Log jets, ordinary jets, and strata.} Every log jet on $(X, \M)$ has an underlying ordinary jet, obtained by ``forgetting'' the map on log structures. That is, there is a natural map $$J_m(X, \M) \to J_m(X)$$ of schemes over $X$. (It is the same as the map induced on log jet spaces by the canonical log scheme morphism $(X, \M) \to (X, \O_X^*)$.) Typically this map is neither surjective nor injective: not every ordinary jet underlies a log jet, and those that do may become log jets in more than one way.

When $(X, \M)$ is fine, and further either $\chr k = 0$ or $(X,\M)$ is also saturated, there is a simple abstract description of the natural map $J_m(X, \M) \to J_m(X)$ in terms of the stratification of $X$ introduced in Section~\ref{strataintro}; namely, on each stratum it is an affine bundle map. To see this, we may work locally on $X$, and suppose that $X = \Spec A$ has a chart by a fine monoid $P$. Recall that the data of a $k$-valued log $m$-jet on $X$ is then a diagram \begin{center} $\xymatrix{ k \times (1+ t k \llbracket t \rrbracket / (t^{m+1})) \ar[d] & P \ar[l]\ar[d] \\ k\llbracket t \rrbracket / (t^{m+1}) & A \ar[l] \\ } $ \end{center} where the top arrow is a monoid morphism and the bottom arrow corresponds to the underlying ordinary jet $\gamma: j_m \to X$. For simplicity of exposition we will think of $\gamma$ as having a fixed closed point $x$ as its image. But our argument following actually makes sense for the generic point $\xi$ of the component of a stratum in which $x$ lies.

The prime ideal $I$ of $P$ consisting of elements which are not units under $P \to \O_{X,x}$ must map to pairs in $k \times (1 + t k\llbracket t\rrbracket / (t^{m+1})$ which are not units, that is, pairs whose first component is zero. So, the underlying ordinary jet lies inside the closed subscheme of $X$ cut out by the elements $\alpha(I) \subseteq \alpha(P)$ which vanish at $x$, and the additional data of the log jet is the map to the second factor $1 + t k\llbracket t \rrbracket / (t^{m+1}).$

If the chart $P$ is sharp at $x$ (in the sense that $I = P - \{1\}$, for example, if $P$ is good near $x$), then this additional data may be specified arbitrarily. That is, a log $m$-jet at $x$ is just a pair consisting of an ordinary $m$-jet $\gamma$ at $x$ which lies (generically) inside the stratum to which $x$ belongs together with a map $P \to 1 + t k \llbracket t \rrbracket / (t^{m+1})$. Since the latter data is equivalent to a {\em group} morphism $$P^{gp} \to 1 + t k\llbracket t \rrbracket / (t^{m+1})$$ it is parametrised by affine space, and the claim follows.

Slightly more generally, for a map of log schemes $j_m \to X$ with given underlying ordinary jet $\gamma$, the map on log structure sheaves is exactly a monoid morphism $$\phi: \M_x \to k \times (1 + t k\llbracket t \rrbracket / (t^{m+1}))$$ which annihilates the maximal proper ideal $\M_x - \M_x^*$ of $\M_x$ on composition with the map to $k \llbracket t \rrbracket/(t^{m+1})$ and is $\O_{X,x}^*$-equivariant (where $\O_{X, x}^*$ acts on the target through $\gamma$). Since the target of $\phi$ is a group it is the same to replace $\M_x$ by $\M_x^{gp}$. Assume either that $k$ has characteristic zero, or that $k$ has characteristic $p > 0$ and $\M_x^{gp}/\O_{X,x}^*$ has no $p$-torsion. Then the torsion part $T$ of $\M_x^{gp}/\O_{X,x}^*$ has order relatively prime to the order of the torsion part of $1 + t k \llbracket t \rrbracket/(t^{m+1})$, so a map $$\overline{\phi}: \M_x^{gp}/\O_{X,x}^* \to 1 + t k \llbracket t \rrbracket / (t^{m+1})$$ factors through the quotient by $T$, hence uniquely determines a map $\phi$ given a chosen splitting of $\M_x^{gp} \to (\M_x^{gp}/\O_{X,x}^*)/T$. Now $\overline{\phi}$ is arbitrary, and parametrised by affine space $\A^{mj}$, where $j = \rank \M_x^{gp}/\O_{X, x}^*$ is the rank of $\M$ at $x$. 

\begin{prop}\label{logjetbundle} (\cite{KS} 3.2) Let $(X, \M)$ be a fine log scheme and, for $m \geq 0$ or $m = \infty$, write $J_m(X, \M)_j$ for the space of log $m$-jets or log arcs over the rank $j$ stratum $X_j$. Then: \begin{itemize} \item[(1)] If $\chr k = 0$ or $\chr k = p$ and $\M^{gp}/\O_X^*$ has no $p$-torsion on $X_j$, the natural map $J_m(X, \M)_j \to J_m(X_j)$ is an affine bundle map with fibre $\A^{mj}$. \item[(2)] The natural map $J_\infty(X, \M)_j \to J_\infty(X_j)$ is an affine bundle map with fibre the space of maps $\Z^j \to 1 + t k\llbracket t \rrbracket.$ \end{itemize} \end{prop}

\begin{proof} The above discussion, which is the case $m \geq 0$, applies with the usual notational changes to the case $m = \infty$. In this case $1 + t k \llbracket t \rrbracket$ has no torsion, independent of the characteristic of $k$, so we do not need an additional hypothesis in (2). \end{proof}

\begin{corollary}\label{logjetdimcount} Assume the notation and hypotheses of Proposition~\ref{logjetbundle}(1). Then $$\dim J_m(X, \M)_j = \dim J_m(X_j) + mj,$$ and, writing $\dim J_m(X, \M)_j^{sm}$ for the log $m$-jets over the smooth locus of $X_j$, $$\dim J_m(X, \M)_j^{sm} = (m+1)(j + \dim X_j) - j.$$ \qed \end{corollary}

\begin{remark}\label{positivecharbundle} In the case where $\chr k = p$ and $\M_x^{gp}/\O_{X,x}^*$ has a $p$-power-torsion subgroup $T_p$, there is in addition to $\overline{\phi}$ some data to a log $m$-jet given by a map $T_p \to 1+ t k\llbracket t \rrbracket / (t^{m+1})$. When $m$ is large enough compared to $n$ the truncation map $J_m(X, \M)_j \to J_n(X, \M)_j$ has in its image only log $n$-jets corresponding to the trivial map on $T_p$, because under the map $$1 + t k\llbracket t \rrbracket / (t^{m+1}) \to 1 + t k\llbracket t \rrbracket / (t^{n+1})$$ if $m-n \geq p$ an element not $1$ has its torsion order decreased. This gives an alternate explanation of Proposition~\ref{logjetbundle}(2) in the case of positive characteristic. \end{remark} 

\begin{example} Here is an illustration of this in terms of differentials in a simple case. Let $(X, \M)$ be the affine line $\Spec k[x]$ with its standard structure. Both $J_1(X, \M)$ and $J_1(X)$ are trivial $\A^1$-bundles over $X$. The former we may give co-ordinates $d_0 x = x$ and $\partial_1 x$, and the latter, $d_0 x$ and $d_1 x = d_0 x \cdot \partial_1 x$. In other words, the map $J_1(X, \M_X) \to J_1(X)$ here is one chart $$k[d_0 x, d_1 x] = k[d_0 x, d_0 x \cdot \partial_1 x] \hookrightarrow k[d_0 x, \partial_1 x]$$ of the blow-up of an affine plane at its origin, (which is the zero jet at $x = 0$). All of the log jets at $x = 0$ map to the zero jet in $J_1(X)$.

Similar takes place for the log $m$-jets of any monoid algebra $X = \Spec k[P]$ with its standard structure, because the monoid morphism $P \to k \times (1 + t k \llbracket t \rrbracket / (t^{m+1}))$ giving the map on log structures alone already determines the underlying ordinary jet $k[P] \to k \llbracket t \rrbracket / (t^{m+1})$. \end{example} 

\section{Log \'etale maps}

The notions of unramified, smooth, and \'etale maps in the log category may be defined by infinitesimal lifting properties in the same way as in the ordinary scheme-theoretic sense. That is, $(X, \M_X) \to (Y, \M_Y)$ is log unramified, resp. log smooth, resp. log \'etale if whenever one has a diagram \begin{center} $ \xymatrix{ (T, \M_T) \ar[r] \ar[d] & (X, \M_X) \ar[d] \\ (T', \M_{T'}) \ar[r] \ar@{.>}[ur] & (Y, \M_Y) \\ } $ \end{center} where $(T, \M_T) \to (T', \M_{T'})$ is an infinitesimal thickening, locally on $T'$ there is at most one, resp. at least one, resp. exactly one lift $(T', \M_{T'}) \to (X, \M_X)$ making a commutative diagram. 

What one must supply to make this work is the correct notion of an infinitesimal thickening of log schemes. Following Kato \cite{Kato} we say that $i: (T, \M_T) \to (T', \M_{T'})$ is a log infinitesimal thickening if (1) the underlying map $T \to T'$ of schemes is a an infinitesimal thickening in the ordinary sense, and (2) the map $i$ is strict, that is, $\M_T = i^* \M_{T'}$.

\begin{remark} The strictness condition (2) appears already in the notion of a closed embedding of log schemes \cite{Kato}.\end{remark}

\begin{prop}\label{jetthickening} For all $0 \leq n \leq m$ the truncation maps $$\pi^m_n: (S \times_k j_n,\O_S \oplus_{k^*} \O_{j_n}^*)  \to (S \times_k j_m, \O_S \oplus_{k^*} \O_{j_m}^*)$$ are log infinitesimal thickenings. \qed \end{prop}
\begin{proof} Obviously condition (1) above is satisfied. The condition (2) is that $\O_S \oplus_{k^*} \O_{j_m}^*$ as a monoid sheaf on $S \times_k j_n$ generates the log structure $\O_S \oplus_{k^*} \O_{j_n}^*$. This is immediate because the sheaf $\O_{j_m}^*$ on $j_n$ generates the trivial structure $\O_{j_n}^*$. \end{proof}

Our special interest in log \'etale maps arises from the following proposition, which is a formal analogue of the same fact for ordinary \'etale maps and ordinary jet and arc spaces.

\begin{prop}\label{etalebasechg1} Let $f: (X, \M_X) \to (Y, \M_Y)$ be a log \'etale map of log schemes. Then the induced maps $f_m: J_m(X, \M_X) \to J_m(Y, \M_Y)$ for $m \geq 0$ or $m = \infty$ make the diagram \begin{center} $ \xymatrix{ J_m(X, \M_X) \ar[r] \ar[d] & J_m(Y, \M_Y) \ar[d]\\ X \ar[r] & Y\\ } $ \end{center} a fibre square.  \end{prop}

\begin{proof} For $m \geq 0$ this follows from the functorial characterisation of log jet spaces and the definition of log \'etale maps. More precisely, an $S$-point of $X \times_Y J_m(Y, \M_Y)$ is naturally a diagram \begin{center} $ \xymatrix{S \ar[r] \ar[d] & X \ar[d] \\ S\times_k j_m \ar[r] & Y \\ } $\end{center} (where we have omitted the log structures from the notation) because of the natural correspondence $$\Hom(S, J_m(Y, \M_Y)) = \Hom_{log}(S \times_k j_m, Y).$$ But if $X \to Y$ is log \'etale then such diagrams biject naturally with lifts $S \times_k j_m \to X$, which are the same things as $S$-points of $J_m(X, \M_X)$. 

Given this, the assertion of the proposition in the case $m = \infty$ follows on taking the projective limit of the maps on log jet spaces.\end{proof}

% \begin{remark}\label{smoothsurjects1} The diagram in the proof obviously shows at the same time that if $(X, \M_X) \to (Y, \M_X)$ is log smooth then $J_m(X, \M_X) \to J_m(Y, \M_Y)$ surjects if $X \to Y$ does.\end{remark}

% \begin{remark}\label{smoothsurjects2} If $(X, \M_X)$ is log smooth over $\Spec k$ then the truncation maps $J_m(X, \M_X) \to J_n(X, \M_X)$ surject. For in the diagrams \begin{center} $\xymatrix{ S \times_k j_n \ar[r]\ar[d] & X \ar[d] \\ S\times_k j_m \ar[r] & \Spec k\\ } $ \end{center} (where we have omitted the log structures from the notation) one has liftings if $X$ is log smooth. \end{remark}

In order to apply such a lemma it is a convenience to have more practical criteria for log \'etaleness. The one we shall need is the observation that the ordinary \'etale maps which are also log \'etale are the strict ones:

\begin{prop}\label{ordetalestrict} (\cite{Kato} 3.8) Let $f: (X, \M_X) \to (Y, \M_Y)$ be a morphism of log schemes whose underlying map $X \to Y$ is \'etale. Then $f$ is log \'etale if and only if $f$ is strict (i.e., $f^* \M_Y \isom \M_X)$. \end{prop}

\begin{proof} In an infinitesimal lifting situation one has a unique lift of the underlying map on schemes, and therefore a unique lift as log schemes if and only if the map on log structures is uniquely determined. This happens if and only if $f$ is strict. \end{proof}

\section{Log arcs for monoid algebras}

We briefly consider the case of monoid algebras. The situation here is as simple as it could be: if a monoid algebra is irreducible, so is its log arc scheme.

\begin{prop}\label{monoidalgebra} Let $P$ be a monoid, not necessarily finitely generated or integral, $X = \Spec k[P]$ its monoid algebra, with the standard log structure $\M$ generated by $P$, and $P^{gp}$ the group completion of $P$. Then: \begin{itemize} \item[(1)] For $m \geq 0$ or $m = \infty$, the $k$-rational points of $J_m(X, \M)$ are the trivial bundle over $X$ with fiber the space of maps $P^{gp} \to 1 + t k\llbracket t\rrbracket/(t^{m+1})$.
\item[(2)] Assume $m = \infty$, or else $m\geq 0$ and either $\chr k = 0$ or $\chr k = p$ and $P$ is $p$-power-saturated, (that is, if $f \in P^{gp}$ has $f^p \in P^{int}$ then in fact $f \in P^{int}$). If $P^{gp}$ has finite rank, the fiber of $J_m(X,\M) \to X$ is irreducible, hence $J_m(X, \M)$ is irreducible if $X$ is. \end{itemize} \end{prop}

\begin{proof} The data of a $k$-valued log jet or arc on $X$ consists of a diagram \begin{center} $ \xymatrix{ k \times (1+ t k\llbracket t \rrbracket/(t^{m+1})) \ar[d] & P \ar[d]\ar[l]_(.2){\alpha\times\beta} \\ k\llbracket t \rrbracket/(t^{m+1}) & k[P]\ar[l]\\ } $ \end{center} which is completely determined by the top arrow $\alpha \times \beta$. Now $\alpha$ is a point of $X$ and $\beta$ is a map $P \to 1+ tk\llbracket t \rrbracket/(t^{m+1})$, equivalently a map $$\beta: P^{gp} \to 1 + tk\llbracket t \rrbracket/(t^{m+1}).$$ So $J_m(X, \M)$ is the bundle of such maps $\beta$. This gives the claim (1).

For (2), the maps $\beta$ factor through the quotient of $P^{gp}$ by its torsion in every case except when $\chr k = p$, $m \geq 0$, and $\M^{gp}/\O_X^*$ has $p$-torsion at some stalk. But the presence of $p$-torsion is what is ruled out by the saturation assumption on $P$ in this case. So the maps $\beta$ are just from a free group, and the claim follows.\end{proof}

In the case that the monoid $P$ is finitely generated, but $\Spec k[P]$ has a non-standard log structure generated by only a subset of its elements, in general the fibers of its log jet and log arc spaces will vary. 

\begin{prop}\label{finitemonoidalgebra} Let $P$ be a fine monoid, $X = \Spec k[P]$ its monoid algebra over a perfect field $k$. Let $B = \{x_1, ..., x_r, z_1, ..., z_s\}$ be a generating set for $P$, so that $k[P]$ is a quotient of $k[B]$ by some monomial relations, and let the (non-standard) log structure $\M$ on $X$ be generated by the elements $x_1, ..., x_r$. Assume that $X$ is irreducible and that the open set $U = (z_1 z_2 \cdots z_s \neq 0)$ meets every component of every stratum $X_j$ of $X$. Then $J_\infty(X, \M)$ is irreducible. \end{prop}

\begin{proof} By hypothesis, $U \cap X_j$ is dense in $X_j$ for each $j$. By Kolchin's theorem, the ordinary arcs on $X_j \cap U$ are then dense in the ordinary arc space of $X_j$. In view of Proposition~\ref{logjetbundle}, it is now enough to show that the log arcs over $U$ are irreducible. But after deleting the locus $z_1 z_2 \cdots z_s = 0$ what remains is isomorphic to $X$ with its standard structure as a monoid algebra, so this follows from the last proposition. \end{proof}

This situation can occur for example when $\Spec k[P]$ is given a non-standard log structure arising from a face $F$ of $P$ generated by elements $x_1, ..., x_r \in F$.

%\begin{example} It is not especially difficult for the hypothesis of this proposition to fail. For example, consider the singular quadric cone $X = \Spec k[x, y, z]/(yz - x^2)$ with log structure that includes $x$ and $y$ but omits $z$. Here the origin is the rank two stratum, and it is included in the locus $z = 0$. \end{example}

\section{Irreducibility over fine log schemes}

Here is a consequence of Proposition~\ref{etalebasechg1}, for when a map of log schemes is both log \'etale and \'etale in the ordinary sense:

\begin{prop}\label{etalebasechg2} Let $f: (X, \M_X) \to (Y, \M_Y)$ be a log \'etale map of log schemes, $f_m: J_m(X, \M_X) \to J_m(Y, \M_Y)$ the induced map on log jet or log arc spaces for $m \in \N$ or $m = \infty$. Then: \begin{itemize} \item[(1)] If $f$ is \'etale (in the ordinary sense) then so is each $f_m$. 
\item[(2)] If $X$ and $Y$ are irreducible, and $f$ is \'etale and surjective, then $J_m(X, \M_X)$ is irreducible if and only if $J_m(Y, \M_Y)$ is. \end{itemize} \end{prop}

\begin{proof} (1) follows from Proposition~\ref{etalebasechg1}, which says that $f_m$ is a base change of $f$. In (2), the base changes $f_m$ are likewise \'etale and surjective. It's immediate that $J_m(Y, \M_Y)$ is irreducible if $J_m(X, \M_X)$ is. Conversely if $J_m(X, \M_X)$ is reducible then it has a component $Z$ lying over some proper closed subset $X_1 \subseteq X$. The image of an open subset $U$ of $Z$ is then open in $J_m(Y, \M_Y)$ and lying over the proper closed subset $Y_1 = \bar{f(X_1)} \subseteq Y$. But if $J_m(Y, \M_Y)$ is irreducible it has no open subsets contained in the fibre over $Y_1$. So $J_m(Y, \M_Y)$ is reducible. \end{proof}

According to Proposition~\ref{ordetalestrict}, the following lemma does not concern a log \'etale map $f$ (it is \'etale in the ordinary sense but not strict), although we might compare the condition appearing here to that in (\cite{Kato}, 3.4-5) for a morphism to be log \'etale.

\begin{lemma}\label{smallstrucchange} Let $k$ be a field of characteristic zero. Let $\phi: \M \to \M'$ be a map of fine log structures on $X$, with $f: (X, \M') \to (X, \M)$ the corresponding map of log schemes. (That is, the underlying map on schemes is the identity on $X$.) Assume that the induced map $$\phi_x^{gp}: \M_x^{gp} \to \M'^{gp}_x$$ at some point $x$ has finite kernel and cokernel. Then near $x$, for $m \geq 0$ or $m = \infty$ the maps $$f_m: J_m(X, \M) \to J_m(X, \M')$$ are isomorphisms. \end{lemma}

\begin{proof} The map $\phi_x^{gp}$ has the same kernel and cokernel as the group completion of the map $$\overline{\phi}_x: P = \M_x/\O_{X, x}^* \to Q= \M_x'/\O_{X, x}^*,$$ since $\phi_x$ is an isomorphism on the units $\O_{X, x}^*$ at the stalks of the two log structure sheaves. So by hypotheses we have a map $P \to Q$ such that the induced map $P^{gp} \otimes_\Z k \to Q^{gp} \otimes_\Z k$ is an isomorphism. 

It follows that the log Hasse-Schmidt algebras $HS^m_X(\M)$ and $HS^m_X(\M')$ calculated from $P$ and from $Q$ are equal. For the map $P^{gp} \otimes_\Z k \to Q^{gp} \otimes_\Z k$ is given by an integer matrix invertible over $k$. Since the first-order log differential $\partial_1$ is a monoid morphism from $P$ or $Q$ to $P^{gp} \otimes_\Z k$ or $Q^{gp} \otimes_\Z k$, we see that we get the same first log Hasse-Schmidt algebra $HS^1_X(\M) = HS^1_X(\M')$. In view of Remark~\ref{logleibnizeqns}, we then have $HS^m_X(\M) = HS^m_X(\M')$ for $m \geq 1$ by induction. \end{proof}

\subsection{Dimensional regularity.} The following condition will appear in our discussion of reducibility and irreducibility for log arc spaces. Recall that the strata $X_j$ of a fine log scheme $(X, \M)$ were introduced in Section~\ref{strataintro}.

\begin{definition}\label{dimreg} We call a fine log scheme $(X, \M)$ of pure dimension {\em dimensionally regular} if there is a number $r$, necessarily $r \geq 0$, such that for all non-empty strata $X_j$ of $X$ one has $\codim_X X_j = j - r$. The number $r$ is then the smallest $j$ such that $X_j$ is non-empty. \end{definition}

\begin{remark} If $X$ is reducible, there is an induced log structure on any component $i: Z \to X$ by pulling back $\M$ to $Z$. Now if $X$ is connected and has pure dimension, then $(X, \M)$ is dimensionally regular if and only if $(Z, i^* \M)$ is dimensionally regular for each component $Z$ of $X$. \end{remark}

Assume for now that $X$ is integral. The sheaf $I$ of sections of $\M$ that map to zero in $\O_X$ forms a sheaf of prime ideals of $\M$. The map $\M \to \O_X$ factors through the quotient $\M/I$, which may be identified with the monoid $\M - I$ together with a zero element corresponding to the class $I$. Now $X_r$ is open and dense in $X$, where $r$ is the minimum rank of $\M$ on $X$, and $\M - I$ has rank zero on $X_r$. So if $X$ is dimensionally regular the number $r$ is just the height of the zero ideal $I$ (in the sense of the height of prime ideals of a monoid, i.e., the codimension of the complementary face $\M - I$) at any point.

So when $X$ is integral there would not any loss from our point of view in just asking for $r = 0$, in other words working with $\M - I$ rather than $\M$, since according to (the proof of) Proposition~\ref{logjetbundle} the difference is just the multiplication of the log jet spaces $J_m(X, \M)$ everywhere by affine factors $\A^{mr}$. On the other hand it is convenient to allow $r > 0$ in general, because of the easy induction it allows. For example:

\begin{prop}\label{dimregchar} Let $(X, \M)$ be a fine log scheme, let $r$ be the minimum rank of $\M$ on $X$, and let $Z(\M) = \cup_{j > r} X_j \neq \emptyset$ be the locus in $X$ where the log structure has rank greater than $r$. Then the following are equivalent: \begin{itemize} \item[(1)] $(X, \M)$ is dimensionally regular;
\item[(2)] $(Z(\M), i^* \M)$ is dimensionally regular and $X_{r+1}$ is non-empty;
\item[(3)] $(Z(\M), i^* \M)$ is dimensionally regular and $Z(\M)$ is the closure of $X_{r+1}$;
\item[(4)] the set of indices $j$ for which the strata $X_j$ are non-empty is an interval $[a,b]$ in $\N$ and for all $a \leq j \leq \ell \leq b$ the stratum $X_\ell$ is contained in the closure of $X_j$.\end{itemize} \end{prop}

\begin{proof} The rank of $i^*\M$ on $Z(\M)$ is the same as the rank of $\M$ on $Z(\M)$ as a subset of $X$. Since $Z(\M)$ is not empty, it has (pure) codimension one, and $X$ is dimensionally regular with minimum rank $r$ if and only if $Z(\M)$ is dimensionally regular with minimum rank $r+1$ along all its strata components of codimension one. This gives the equivalence of (1) with (2) and (3). The characterisation (4) follows from the equivalence of (1) and (3) by induction. \end{proof}

In addition to the conditions here, which are essentially in terms of the combinatorial relationships of the strata, we give another, more geometric characterisation in Proposition~\ref{dominantchart}.

\subsection{Dimensional regularity in log geometry.} Before proceeding we briefly place our condition of dimensional regularity in context. As a first observation, a toric variety with its standard log structure is dimensionally regular, with $r = 0$, because the strata of a toric variety are its torus orbits and the rank along each orbit is its codimension.

\begin{definition} (Kato \cite{KatoTS}) One may define the strata of $(X, \M_X)$ locally scheme-theoretically at point $x$ by taking the ideal $I(\M_X, x)$ generated by the elements of $\alpha(\M_{X, x})$ which vanish at $x$. Then one says that $(X, \M_X)$ is {\em log regular} if it is dimensionally regular with constant $r = 0$ and the strata, so defined scheme-theoretically, are regular. \end{definition}

If $(X, \M)$ is a fine saturated log smooth scheme then it is log regular. Kato showed that over a perfect field the converse holds.

\begin{prop}\label{katoregular} (\cite{KatoTS} 8.3) Let $(X, \M_X)$ be a fine and saturated scheme over $(\Spec k, k^*)$, where $k$ is a perfect field. Then $X$ is log smooth if and only if it is log regular.\qed \end{prop}

This obviously subsumes the fact that toric varieties are dimensionally regular.

Dimensional regularity, however, is a much weaker condition than log smoothness.

\begin{example} Let $X$ be the affine plane $\Spec k[x,y]$ with non-standard structure generated by the monomial $xy$ (see Example~\ref{a2nstdex}). Then $X_1 = \Spec k[x,y]/(xy)$ is singular. So $X$ is not log smooth. But it is dimensionally regular. \end{example}

\begin{example} Consider the affine line $X = \Spec k[t]$ with log structure given by the map $k[x,y] \to k[t]$ which takes both generators $x, y$ to $t$. This is the induced structure on the diagonal $x = y$ of the plane $\Spec k[x,y]$. The rank of the log structure on $X$ is zero away from the origin but jumps to two at $t = 0$, so this log scheme is not dimensionally regular.\end{example}

%One does not see the failure of log smoothness in this example by just trying to lift log jets, since in fact here the projection maps on log jet spaces are surjective. Instead consider the diagram \begin{center} $ \xymatrix{ (\Spec k, \N\epsilon) & (\Spec k[t], \N^2) \ar[l] \\ (\Spec k[\epsilon]/(\epsilon^2), \N\epsilon) \ar[u] & (\Spec k, k^*) \ar[l] \ar[u] \\ } $ \end{center} where three of the log schemes have charts as indicated, and the top map sends $t \mapsto 0$ and generators $x, y$ of $\N^2$ to $a\epsilon, b\epsilon \in \N \epsilon$. This makes sense because the log structure of $(\Spec k, \N \epsilon)$ is pulled back from $(\Spec k[\epsilon]/(\epsilon^2), \N \epsilon),$ so the generator $\epsilon$ maps to zero in $k$. A lift $(\Spec k[t], \N^2) \to (k[\epsilon]/(\epsilon^2), \N \epsilon)$ would have $t \mapsto c\epsilon$ for some $c \in k$ at the level of co-ordinate rings, but then at the level of log structures we must have both $c = a$ and $c = b$, which is impossible by choice. \end{example}

\begin{example}\label{threelines} Take $\Spec k[x,y]$ with log structure generated by $x, y,$ and $x-y$. So the log structure is supported on the three lines $x = 0, y = 0, x = y$, on which it has rank one, except at the origin, where it has rank three. Similar to the previous example, this can be realised as the induced structure on the plane $z = x-y$ in $\Spec k[x,y,z]$ with its standard structure. \end{example}

\begin{prop}\label{dominantchart} Let $(X, \M)$ be an irreducible fine log scheme with a sharp chart $P$ near a point $x$. For each irreducible component $Z$ of the stratification of $X$ which contains $x$, choose units $u_1, ..., u_{\dim Z} \in \O_{X, x}$ which are algebraically independent in the residue field $\O_{X, Z}/\m_Z$ of $\O_{X,Z}$. Let $Q = P \oplus \N^{\dim Z}$ be the chart $P$ expanded to include these units. This $Q$ is also a chart for $(X, \M)$ over a neighbourhood $U$ of $x$ over which $u_1, ..., u_{\dim Z}$ are defined. 

Then $X$ is dimensionally regular near $x$ of minimum rank $r = 0$ if and only if the chart morphisms $U \to \Spec k[Q]$ for the various components $Z$ through $x$ are dominant. \end{prop}

\begin{proof} Let $j$ be the rank of $\M$ along $Z$, and let $x_1, ..., x_j \in \m_Z$ be the images of generators in $P$ for the log structure on $Z$. 

If $j > \codim_X Z$ then the elements $u_1, ..., u_{\dim Z}, x_1, ..., x_j \in \O_{X,Z}$ number more than the transcendence degree $\dim X$ of the fraction field of $\O_{X, Z}$ over $k$, so satisfy a polynomial equation $G = 0$ with coefficients in $k$. The image of $U \to \Spec k[Q]$ is then contained in the proper locus $G = 0$. 

On the other hand if $j \leq \codim_X Z$ then there is no such polynomial equation $G = 0$ over $k$. For supposing there were, let $d$ be the smallest number such that every term of $G$ lies in $\m_Z^d$, and consider the equation $G = 0$ modulo $\m_Z^{d+1}$. Writing the monomials in $x_1, ..., x_j$ that appear in $G$ modulo $\m_Z^{d+1}$ in terms of a basis of $\m_Z^d / \m_Z^{d+1}$ over the residue field $\O_{X, Z}/\m_Z$, we obtain a polynomial with co-efficients in $\O_{X,Z}/\m_Z$ satisfied by the units $u_1, ..., u_{\dim Z}$, contrary to assumption. \end{proof}

\subsection{The main theorem.} The necessity of dimensional regularity in Theorem~\ref{maintheorem} is a straight-forward consequence of Proposition~\ref{logjetbundle}. Recall that for a fine log scheme $(X, \M)$ we write $J_m(X, \M)_j$ and $J_\infty(X, \M)_j$ for the log $m$-jets and log arcs over the stratum $X_j$.

\begin{prop}\label{maintheoremnec} Let $(X, \M)$ be a fine log scheme with $X$ irreducible. If $(X, \M)$ is not dimensionally regular then $J_\infty(X, \M)$ is reducible. \end{prop}

\begin{proof} We may assume that $X$ is reduced. Let $r$ be the minimum rank of $\M$ on $X$ and let $j > r$ be such that $X_j$ has codimension less than $j - r$. 

Write $J_m(X, \M)_\ell^{sm}$ for the log $m$-jets of $X$ over the smooth locus of a stratum $X_\ell$. According to Corollary~\ref{logjetdimcount}, for $m$ large enough we have $$\dim J_m(X, \M)_j^{sm} > \dim J_m(X, \M)_r^{sm}.$$ Let $C_m \subseteq J_m(X, \M)$ be the image of $J_\infty(X, \M)$ under the natural projection. It is the constructible subset of log $m$-jets which extend to log arcs. If $J_\infty(X, \M)$ is irreducible then so is $C_m$ (as a topological space), and further $C_m$ has $J_m(X, \M)_r^{sm}$ as a dense subset. But $C_m$ also contains $J_m(X, \M)_j^{sm}$, a contradiction. \end{proof}

We turn to establishing the converse implication. We will make use of the following observation:

\begin{prop}\label{kolchinstratum} Let $X_j$ be a stratum of a fine log scheme $(X, \M)$. The irreducible components of $J_\infty(X, \M)_j$ biject with the irreducible components of $X_j$, and the log arcs over the smooth locus $X_j^{sm}$ of $X_j$ are dense in $J_\infty(X, \M)_j$. \end{prop}

\begin{proof} This follows from Kolchin's theorem for the ordinary jets on the stratum $X_j$, together with the characterisation of the natural map $J_\infty(X, \M)_j \to J_\infty(X_j)$ as a bundle map. \end{proof}

Perhaps tendentiously, this suggests the strategy of describing the components of $J_\infty(X, \M)$ by comparing the arcs on one level of the stratification with the arcs on the smooth locus of the next. The simplest situation (apart from the base case $X = X_r$, where the result is just Kolchin's theorem) is in the following lemma. It will be the inductive step of our argument.

\begin{lemma}\label{basicstep} Let $(X, \M)$ be a fine log scheme with $X$ irreducible and such that $X = X_r \cup X_s$ for some $s > r$, (with $X_r, X_s$ non-empty). Then $J_\infty(X, \M)$ is irreducible if and only if $s = r + 1$. \end{lemma}

\begin{proof} The necessity of the condition $s = r + 1$ is a special case of Proposition~\ref{maintheoremnec}. We consider the converse implication. We may assume that $X$ is reduced.

We may replace $X$ by an open subset which meets $X_s$ and does not meet the singular loci of $X_r$ and $X_s$, and thereby assume first that the singular locus of $X$ lies inside $X_s$ and second that $X_s$ is smooth. For by Proposition~\ref{kolchinstratum} the arcs over this open subset are dense in $J_\infty(X, \M)$. We may further replace the sheaf $\M$ by the complementary face of its zero ideal, and hence assume that $r = 0$ and $s = 1$.

By assumption, either $X$ is smooth or it is singular along $X_1$. Consider the normalisation $\tilde{X} \to X$, with log structure $\tilde{\M}$ pulled back from $X$. This is an isomorphism over the smooth locus $X_0$, and after deleting from $X$ a (positive-codimension) subset of $X_1$ we may assume that $\tilde{X}$, which was non-singular in codimension one, is in fact smooth, and that the normalisation is unramified, hence \'etale, over $X_1$. In particular, the map $J_\infty(\tilde{X}, \tilde{\M}) \to J_\infty(X, \M)$ surjects. Thus in any case we are reduced to the situation where $X$ is smooth.

Now let $\M'$ be the log structure on $X$ generated by local equations for the divisor $X_1$. Near a component $Z$ of $X_1$ with generic point $\xi$ this has a chart $\O_{X,\xi}/\O_{X,\xi}^* \isom \N$ and the local valuation maps $$\M_\xi \to \M_\xi/\O_{X, \xi}^* \to \O_{X, \xi} / \O_{X,\xi}^*$$ determine a morphism $\M \to \M'$ of log structures on $X$. By construction, Lemma~\ref{smallstrucchange} applies to this morphism, and we may replace $\M$ by $\M'$.

In particular we now are reduced to having $(X,\M)$ fine, saturated, and log regular. According to Proposition~\ref{katoregular} this means that $X$ is log smooth, which as one may expect is sufficient for the conclusion we seek. However, a short calculation finishes the argument in any case. For now locally near $Z \subseteq X_1$ there is the chart $X \to \Spec k[x]$, where $x$ is a local equation for $Z$, a smooth morphism. So there is locally an \'etale map from $X$ to some affine space $Y = \Spec k[x, z_1, ..., z_t]$ over $\Spec k[x]$, with log structure generated by $x$. By Proposition~\ref{finitemonoidalgebra}, or else by an easy direct computation, the log arc (and log jet) spaces of $Y$ with this log structure are irreducible, and the log arcs over the complement of the hyperplane $x = 0$ are dense. An \'etale map of schemes which is strict is log \'etale, so the induced map on log arc spaces is also \'etale. We are in the situation of Proposition~\ref{etalebasechg2}, and the claim follows.\end{proof}

We are ready to argue for Theorem~\ref{maintheorem}. 

\begin{prop}\label{irredcrit} Let $(X, \M)$ be a fine log scheme with $X$ irreducible. Then $J_\infty(X, \M)$ is irreducible if and only if $(X, \M)$ is dimensionally regular. \end{prop}

\begin{proof} After Proposition~\ref{maintheoremnec} it remains to show that if $(X, \M)$ is dimensionally regular then $J_\infty(X, \M)$ is irreducible. 

Let $r$ be the minimum rank of $\M$ on $X$. By Lemma \ref{basicstep} applied to the components of $$X_r \cup X_{r+1} = X - \cup_{j > r+1} X_j,$$ the log arcs $J_\infty(X, \M)_r$ are dense in $J_\infty(X, \M)_{r+1}$. Since $\cup_{j > r+1} X_j$ is contained in the closure of $X_{r+1}$, by Proposition~\ref{dimregchar}, by induction $J_\infty(X, \M)_{r+1}$ is dense in $\cup_{j \geq r + 2} J_\infty(X, \M)_j$. Therefore the irreducible $J_\infty(X, \M)_r$ is dense in $J_\infty(X, \M)$. So $J_\infty(X, \M)$ is irreducible. \end{proof}

\subsection{Remarks on positive characteristic.}\label{positivechar}  Let us consider briefly the case where $k$ has  $\chr k = p > 0$. Kolchin's theorem holds for all varieties over $k$ if and only if $k$ is perfect \cite{NS}, so we assume that $k$ is perfect. We easily see that the conditions of Theorem~\ref{maintheorem} are no longer sufficient:

\begin{example}\label{positivecharexample} Let $X = \Spec k[x]$ with log structure $\M$ generated by $x^p$. Away from $x = 0$ the $k$-valued log arcs are the ordinary arcs $x(t) = d_0 x + d_1 x t + ...$ with $d_0 x \neq 0$, which then have $$x(t)^p = (d_0 x)^p + (d_1 x)^p t^p + ... = (d_0 x)^p (1 + (\partial_1 x)^p t^p + ...).$$ The limits of these at $x = 0$ are log arcs, underlain by the zero ordinary arc, in which only powers of $t^p$ appear in the principal series $x^p(t)$. But a log arc at the origin corresponds to an arbitrary series $x^p(t) = 1 + \partial_1 (x^p) t + ....$ So the log arcs over the locus $x \neq 0$ are not dense in $J_\infty(X, \M)$. \end{example}

%In this example, the chart morphism $f: \Spec k[x] \to \Spec k[x^p]$ is inseparable, and not smooth in the ordinary sense. Since it is strict, it is not log smooth; and we see that the induced map on log arc spaces does not surject. 

Here is one way to look at this. There is the map $$g: (\Spec k[x], x) \to (\Spec k[x], x^p)$$ to $X$ from the affine line with its standard structure. This is modelled on the map $p\N \to \N$ on charts, whose corresponding map on monoid algebras is the inseparable map $\Spec k[x] \to \Spec k[x^p]$. Now the induced map $p\Z \to \Z$ on group completions does not become an isomorphism on tensoring with $k$. So although $g$ is an isomorphism away from the origin $x = 0$, it does not induce a surjection of log arc spaces at the origin.

%In positive characteristic, the correct formulation of the criteria Proposition~\ref{etalecriterion}, Remark~\ref{smoothcriterion}, and Lemma~\ref{smallstrucchange} has in addition to the finiteness conditions on (torsion parts of) kernels and cokernels the requirement that these groups also have order not divisible by $p$. What Example~\ref{positivecharexample} shows is that this additional condition is required in some form for our irreducibility theorem to hold in positive characteristic.

\begin{remark} Aside from this, our proof of Theorem~\ref{maintheorem} fails in general in positive characteristic at the normalisation step in Lemma~\ref{basicstep}. Again the issue is of having log structure on a locus that gives rise to an inseparable map. 

Of course one may give a sufficient criterion just by assuming the conclusions these steps were meant to obtain. So let $(X, \M)$ be a fine log scheme over a perfect field $k$ with $\chr k = p$. Assume \begin{itemize} \item[(1)] $(X, \M)$ is dimensionally regular of minimum rank $r$, 
\item[(2)] for each $j \geq r$, the scheme $\overline{X_j}$ is non-singular in codimension one, and
\item[(3)] for each $j \geq r$ and each generic point $\xi$ of the codimension one locus $X_{j+1} = \overline{X_j} - \overline{X_{j+2}}$ of $\overline{X_j}$, the valuation map $$\phi: \M_\xi/\O_{X, \xi}^*-I \to \O_{X, \xi}/\O_{X, \xi}^* \isom \N$$ (where $I$ is the zero ideal of $\M_\xi/\O_{X, \xi}^*$) is such that $\phi^{gp} \otimes k$ is an isomorphism; that is, $\M^{gp}_\xi/\O_{X, \xi}^* - I$ has rank one as an abelian group and no $p$-torsion part, and the image of $\phi$ is not contained in the submonoid $p\N$ of multiples of $p$. \end{itemize} Then $J_\infty(X, \M)$ is irreducible if $X$ is. That a log smooth scheme need not satisfy (2) illustrates that this criterion is not necessary.\end{remark}

\end{document}